\documentclass[11pt,reqno]{amsart}
\usepackage{amsmath,amsthm,amscd,amsfonts,amssymb,color}
\usepackage{cite}
\usepackage[bookmarksnumbered,colorlinks,plainpages]{hyperref}
\newtheorem{theorem}{Theorem}[section]
\newtheorem{lemma}[theorem]{Lemma}

\theoremstyle{definition}
\newtheorem{definition}[theorem]{Definition}
\newtheorem{example}[theorem]{Example}
\theoremstyle{remark}
\newtheorem{remark}[theorem]{Remark}
\numberwithin{equation}{section}
\def\DJ{\leavevmode\setbox0=\hbox{D}\kern0pt\rlap
	{\kern.04em\raise.188\ht0\hbox{-}}D}
\begin{document}
	\title[Convergence of $C-$class Enriched Contraction]{Convergence of $C-$class Enriched Hardy-Rogers Contractions in Some Banach Spaces}
	\author[A.H. Ansari,]{Arsalan Hojjat Ansari$^{1,2.3}$}
	\address{\textbf{A.H. Ansari} $^{1}$ Department of Mathematics and Applied
		Mathematics, Sefako Makgatho Health Sciences University, Ga-Rankuwa,
		Pretoria, Medunsa-0204, South Africa\\
		$^{2}$Department of Mathematics, Karaj Branch, Islamic Azad University,
		Karaj, Iran\\
		$^{3}$Department of Mathematics, Payame Noor University, Tehran, Iran}
	\email{analsisamirmath2@gmail.com,mathanalsisamir4@gmail.com}
	\author[O.J.Omidire,]{Olaoluwa Jeremiah Omidire$^{4}$}
	\address{\textbf{O.J. Omidire}$^{4}$ Department of Mathematical Sciences,
		Osun State University, Osogbo, Nigeria.}
	\email{olaoluwa.omidire@uniosun.edu.ng, omidireolaoluwa@gmail.com}
	\maketitle
	
	\begin{abstract}
		~~	\\
		The applicability of classical Banach contraction mapping principle in solving diverse problems caught the attention of  several researchers in various fields of science and engineering. Since its introduction, many extensions and generalizations of this principle continue to emerge in different directions. Some are seen in form of common fixed point using commuting mappings, A-contractions, rational-type contractions, enriched contractions amongst others. In this paper, {\it{C-class}} function is used to generalizing enriched Hargy-Rogers contractions. Convergence of unique fixed point and common fixed points of these general class of mappings were established using some related iterative schemes. Results obtained are generalizations of some recently announced related results in literature.
	\end{abstract}
	
	\setcounter{page}{1}
	
	\centerline{}
	
	\centerline{}
	
	\bigskip
	
	\section{Background}
	Diverse real life problems in Mathematics, Engineering and several other fields may be formulated as Mathematical equation, such as systems of non-linear equations, optimization theory, differential equation amongst others. These problems may be solved using fixed point approach.
	\par If $\bf{Z}$ is a nonempty set and $P:\bf{Z} \to \bf{Z}.$ Any element $c \in \bf{Z}$ is a fixed point of $P,$ if $P(c)=c$ and denote the set of all fixed points of $P$ by $F_{P}=\{c \in \bf{Z} : P(c)=c\}.$\\
	Given a complete metric space $\bf{Z}$ and a continuous self-map $P$ on $\bf{Z}$ such that
	\begin{equation}
		d(Pc,Pe)\leq a.d(c,e),\text{ \ \ \ \ }\forall~~
		c,e \in \bf{Z}  \label{1}
	\end{equation}%
	with $a\in \lbrack 0,1)$ fixed; Banach\cite{1}, established that $P$ has a unique fixed point in $\bf{Z}.$
	In generalizing this, Rakotch \cite{16} proved that the conclusion of the above claim is valid, if given monotone decreasing function \ $a:(0,\infty )\rightarrow \lbrack 0,1)$ such
	that, for each $c,e\in \bf{Z}, c\neq e$,%
	\begin{equation}
		d(Pc,Pe)\leq a(d(c,e)).  \label{2}
	\end{equation}%
	Also, Kannan\cite{9} proved that $P$ need not be continuous to have a unique fixed point using the inequality:
	$ \exists~~ a\in \lbrack 0,1/2)$ such that%
	\begin{equation}
		d(Pc,Pe)\leq a[d(c,Pc)+d(e,Pe)],\text{\ \ \ \ \ \ \ \ \ \ \ \ \ \ \ \ }%
		\forall c,e\in \bf{Z}.  \label{3}
	\end{equation}%
	Several generalizations and extensions of (\ref{1}) are available in literature. Meanwhile, in 1976, Jungck \cite{8} established the notion of common fixed point of mappings. He proved that the 	pair of commuting self-mappings $S, P$ on a complete metric space $(\bf{Z}, d)$ satisfying
	\begin{equation}
		d(Pc,Pe)\leq ad(Sc,Se),\text{\ \ \ \ \ \ \ }\forall ~~~c,e\in
		\bf{Z}, a\in \lbrack 0,1),  \label{4}
	\end{equation}%
	have a unique common fixed point in $\bf{Z}.$
	
	\begin{definition}
		\cite{ol} \textquotedblleft Let X be a non-empty set. Two mappings $S,T:M\rightarrow M$ are
		said to commute iff $ST=TS$.\textquotedblright
	\end{definition}
	
	\section{method}
	In \cite{Ans1}, the concept of $C-$class functions was introduced.
	\begin{definition}
		\label{def:1.1}\cite{Ans1} \textquotedblleft A mapping $G:[0,\infty )^{2}\rightarrow \mathbb{R}
		$ is called \textit{$C-$class} function if it is continuous and satisfies
		the following axioms:
		
		(1) $G(s,t)\leq s$;
		
		(2) $G(s,t)=s$ implies that either $s=0$ or $t=0$; for all $s,t\in \lbrack
		0,\infty )$.\textquotedblright
	\end{definition}
	
	Note for some $G$ we have that $G(0,0)=0$.\newline
	We denote the set of $C-$class functions by $\mathcal{C}$.
	See \cite{Ans1} for examples.
\begin{definition} \label{def:1.2}\cite{Khan} \textquotedblleft A function $\psi:[0,\infty)\rightarrow [0,\infty)$
		is called an altering distance function if the following properties are
		satisfied:
		
		$(i)$ $\psi $ is non-decreasing and continuous,
		
		$(ii)$ $\psi(t)=0$ if and only if $t=0$.\textquotedblright
\end{definition}
	
	We denote the set of altering distance functions by $\Psi $.
	
	\begin{definition}
		\label{def:1.3}\cite{Ans1} \textquotedblleft Let $\Phi _{u}$ denote the class of the functions 
		$\varphi:[0,\infty)\rightarrow [0,\infty)$ which satisfy the following
		conditions:
		
		$(i)$ $\varphi$ is continuous;
		
		$(ii)$ $\varphi (t)>0,t>0$ and $\varphi (0)\geq 0$.\textquotedblright
	\end{definition}
	
	\begin{definition}
		\label{def:1.4} \cite{Ans1} \textquotedblleft A tripled $(\psi ,\varphi ,G)$ where $\psi \in \Psi ,$ $\varphi
		\in \Phi _{u}$ and $G\in \mathcal{C}$ is said to be monotone if for any $%
		x,y\in \left[ 0,\infty \right) $%
		\begin{equation*}
			x\leqslant y\Longrightarrow G(\psi (x),\varphi (x))\leqslant G(\psi
			(y),\varphi (y)). \textquotedblright
		\end{equation*} 
	\end{definition}
	
	\begin{example}
		\cite{Ans1} \textquotedblleft Let $G(s,t)=s-t,\phi (x)=\sqrt{x}$\textrm{%
			\begin{equation*}
				\psi (x)=%
				\begin{cases}
					\sqrt{x} & \text{if }0\leq x\leq 1, \\ 
					x^{2}, & \text{if x}>1%
				\end{cases}%
				,
			\end{equation*}%
		}then $(\psi ,\phi ,G)$ is monotone.\textquotedblright
	\end{example}
	
	\begin{example}
		\cite{Ans1} \textquotedblleft Let $G(s,t)=s-t,\phi (x)=x^{2}$\textrm{%
			\begin{equation*}
				\psi (x)=%
				\begin{cases}
					\sqrt{x} & \text{if }0\leq x\leq 1, \\ 
					x^{2}, & \text{if x}>1%
				\end{cases}%
				,
			\end{equation*}%
		}then $(\psi ,\phi ,G)$ is not monotone. \textquotedblright
	\end{example}
	
	\begin{lemma}
		\label{lem:2.1} \cite{Chuadchawna} \textquotedblleft Let $E$ be a Hilbert $C^{\ast}$-module
		and $\{x_{n}\} $ be a sequence in $E$ such that $\left\Vert
		x_{n}-x_{n+1}\right\Vert \rightarrow 0$ as $n\rightarrow +\infty $. If $%
		\{x_{n}\}$ is not a Cauchy sequence then there exist an $\varepsilon >0$ and
		sequences of positive integers $\{m(k)\}$ and $\{n(k)\}$ with $m(k)>n(k)>k$
		such that:
		
		$(i)$ $\lim_{k\rightarrow +\infty }\left\Vert
		x_{m(k)-1}-x_{n(k)+1}\right\Vert=\varepsilon $;
		
		$(ii)$ $\lim_{k\rightarrow +\infty }\left\Vert
		x_{m(k)}-x_{n(k)}\right\Vert=\varepsilon ;$
		
		$(iii)$ $\lim_{k\rightarrow +\infty }\left\Vert
		x_{m(k)-1}-x_{n(k)}\right\Vert=\varepsilon;$
		
		$(iv)$ $\lim_{k\rightarrow +\infty }\left\Vert
		x_{m(k)+1}-x_{n(k)+1}\right\Vert=\varepsilon;$
		
		$(v)$ $\lim_{k\rightarrow +\infty }\left\Vert x_{m(k)}-x_{n(k)-1}\right\Vert
		=\varepsilon.$\textquotedblright
	\end{lemma}
	
	\begin{definition}
		\label{def:1.32024+328+1}\cite{ol,chen}  \textquotedblleft Let S and f be two self mappings of
		a nonempty subset M of a normed linear space N. Then $(S,f)$ is a Banach
		operator pair, if any one of the following conditions is satisfied:
		
		$(i)$ $S(Ff)\subseteq Ff$ (i.e) $Ff$ is S-invariant.
		
		$(ii)$ $fSx=Sx$ for each $x\in Ff$ .
		
		$(iii)$ $fSx=Sfx$ for each $x\in Ff$ .
		
		$(iv)$ $||Sfx-fx||\leq k||fx-x||$ for some $k\geq 0.$\textquotedblright
	\end{definition}
	
	\begin{definition}
		\label{def:1.32024+328+2}\cite{ol} \textquotedblleft Let $(N,||.||)$ be a normed linear space.
		A mapping $f:N\rightarrow N$ is said to be an enriched Hardy-Rogers
		contraction if for any $\delta \in \lbrack 0,\infty )$ and non-negative $%
		c_{i}\in \mathbb{R}_+$, ($i=1,2,$\textperiodcentered $\ $\textperiodcentered $\ $%
		\textperiodcentered $\ ,5$) with$\sum_{i=1}^{5}c_{i}<1$, such that $\forall
		u,v\in N$ we have
		\begin{equation}
			||\delta (u-v)+fu-fv||\leq
			c_{1}||u-v||+c_{2}||u-fu||+c_{3}||u-fv||+c_{4}||v-fu||+c_{5}||v-fv||.\textquotedblright
			\label{5} 
		\end{equation}
	\end{definition}
	
	\begin{definition}
		\label{def:1.32024+328+3}\cite{ol} \textquotedblleft Let $(N,||.||)$ be a normed linear space.
		A mapping $f:N\rightarrow N$ is said to be an enriched Jungck-type
		Hardy-Rogers contraction if for any $\delta \in \lbrack 0,\infty )$ and
		non-negative $c_{i}\in \mathbb{R}_+$, ($i=1,2,$\textperiodcentered $\ $%
		\textperiodcentered $\ $\textperiodcentered $\ ,5$) with$%
		\sum_{i=1}^{5}c_{i}<1$, there is a mapping $S:N\rightarrow N$ which commutes
		with operator f, such that $\forall u,v\in N$ we have\textrm{%
			\begin{equation}
				||\delta (Su-Sv)+fu-fv||\leq
				c_{1}||Su-Sv||+c_{2}||Su-fu||+c_{3}||Su-fv||+c_{4}||Sv-fu||+c_{5}||Sv-fv||. \textquotedblright
				\label{6}
			\end{equation}%
		}
	\end{definition}
	\begin{definition} \cite{24,25}
	\textquotedblleft A mapping $S: \mathbb{R}^n \mapsto \mathbb{R}$ is said to be convex if for any $x,y \in \mathbb{R}^n$ and $\aleph \in [0,1]$ we have that
	$$S(x+\aleph(y-x)) \leq \aleph S(x) + (1-\aleph)Sy.\textquotedblright$$
	\end{definition}
	\begin{theorem}\cite{ol} \textquotedblleft Let $(B,||.||)$ be a real Banach space and $f:B\rightarrow B$ a
		convex mapping satisfying definition \ref{def:1.32024+328+2}. Then, f has a
		unique fixed point u in B.\textquotedblright
	\end{theorem}
	
	\begin{theorem}
		\cite{ol} \textquotedblleft Let $(B,||.||)$ be a real Banach space and $(f,S):B\rightarrow B$
		be commuting convex Banach operator pair satisfying definition \ref%
		{def:1.32024+328+3}. If $f(B)\subseteq S(B)$, then $f$ and $S$ have a unique
		common fixed point u in B.\textquotedblright
	\end{theorem}
	\begin{lemma} \cite{22,23}
	\label{lem:1.32024+328+4} \textquotedblleft Given $K$ a convex subset of a normed linear space $N$
	and $f: K \to K.$ Then, $f_c: K \to K$ given by 
	$$f_c(u) = (1 - c)u + cfu ~~~~\forall ~u \in N, ~~c \in [0,1]$$
	is called averaged mapping and it has the following property:
	$Fix (f_c) = Fix (f).$\textquotedblright
	\end{lemma}
	
	\par The $C-$ class function will be used in generalizing enriched Hardy-Rogers contraction and enriched Jungck-type Hardy-Rogers contraction which is a generalization and extension of Banach, Jungck, Kannan, Hardy-Rogers contractive definitions and many more in literature. Also, some iterative procedure together with some basic analysis tools such as continuity, convexity, triangle inequality and many more were employed in establishing our results.
	
	\section{Preliminary Results}
	
	We introduce the following definitions which are generalization of  definitions \ref{def:1.32024+328+2} and \ref{def:1.32024+328+3} and many more in literature.
	
	\begin{definition}
		\label{def:1.32024+328+4}Let $(\bf{Z},||.||)$ be a normed linear space. A mapping 
		$P:\bf{Z}\rightarrow \bf{Z}$ is said to be an enriched Hardy-Rogers $C-$class
		contraction if for any $\delta \in \lbrack 0,\infty )$ and non-negative $%
		c_{i}\in \mathbb{R}_+$, ($i=1,2,$\textperiodcentered $\ $\textperiodcentered $\ $%
		\textperiodcentered $\ ,5$) with$\sum_{i=1}^{5}c_{i}=1$, such that $\forall~
		u,v\in \bf{Z}$ we have\textrm{%
			\begin{eqnarray}
				\psi (||\delta (u-v)+Pu-Pv||) &\leq &G\Big(\psi
				\big(c_{1}||u-v||+c_{2}||u-Pu||+c_{3}||u-Pv||\notag \\
				&+&c_{4}||v-Pu||+c_{5}||v-Pv||\big), \notag
				\label{4n} \\
				&&\varphi
				\big(c_{1}||u-v||+c_{2}||u-Pu||+c_{3}||u-Pv||\notag \\
				&+&c_{4}||v-Pu||+c_{5}||v-Pv||\big)\Big).
			\end{eqnarray}%
		}where $(\psi ,\varphi ,G)\in \Psi \times \Phi _{u}\ \times C$
	\end{definition}
	
	\begin{definition}
		\label{def:1.32024+328+5}Let $(\bf{Z},||.||)$ be a normed linear space. A mapping 
		$P:\bf{Z}\rightarrow \bf{Z}$ is said to be an enriched Jungck-type Hardy-Rogers $C-$%
		class contraction if for any $\delta \in \lbrack 0,\infty )$ and
		non-negative $c_{i}\in \mathbb{R}_+$, ($i=1,2,$\textperiodcentered $\ $%
		\textperiodcentered $\ $\textperiodcentered $\ ,5$) with$%
		\sum_{i=1}^{5}c_{i}=1$, there is a mapping $S:\bf{Z}\rightarrow \bf{Z}$ which commutes
		with operator $P,$ such that $\forall~ u,v\in \bf{Z}$ we have\textrm{%
\begin{eqnarray}
\psi (||\delta (Su-Sv)+Pu-Pv||) &\leq &G\Big(\psi
\big(c_{1}||Su-Sv||+c_{2}||Su-Pu||+c_{3}||Su-Pv||\notag \\
&+&c_{4}||Sv-Pu||+c_{5}||Sv-Pv||\big),\notag\\
&&\varphi
\big(c_{1}||Su-Sv||+c_{2}||Su-Pu||+c_{3}||Su-Pv||\notag \\
&+&c_{4}||Sv-Pu||+c_{5}||Sv-Pv|\big)\Big).\label{5n}
\end{eqnarray}%
}where $(\psi ,\varphi ,G)\in \Psi \times \Phi _{u}\ \times C$
\end{definition}
\begin{remark}
If in \ref{def:1.32024+328+4}\\ $RHS =c_{1}||u-v||+c_{2}||u-Pu||+c_{3}||u-Pv||
+c_{4}||v-Pu||+c_{5}||v-Pv||$ then, it reduces to definition \ref{def:1.32024+328+2} (i.e $G, \psi$ and $\varphi$ are identity operator).\\ 
If in \ref{def:1.32024+328+5},~~~ $RHS =c_{1}||Su-Sv||+c_{2}||Su-Pu||+c_{3}||Su-Pv||
+c_{4}||Sv-Pu||+c_{5}||Sv-Pv||$ then, it reduces to definition \ref{def:1.32024+328+3} (i.e $G, \psi$ and $\varphi$ are identity operator). 
\end{remark}
\section{Main Results}
	
	\begin{theorem}
		\label{t24+1th1}Let $(B,||.||)$ be a real Banach space and $f:B\rightarrow B$
		a convex mapping satisfying definition \ref{4n}. Then, f has a unique fixed
		point $u^{\ast }$ in B.
	\end{theorem}
	
	\begin{proof}
		For a generic point $u_{1}$ in X, we define the sequence $\{u_{n}\}_{n\in \mathbb{N}},$ by Schaefer iteration $$u_n= (1-c)u_{n-1}+cfu_{n-1},~ \forall ~n\geq 1.$$ From (\ref{4n}), let $v=fu$, we have
			\begin{eqnarray*}
				&&\psi (||\delta (u-fu)+fu-f^{2}u||) \\
				&\leq &G(\psi
				(c_{1}||u-fu||+c_{2}||u-fu||+c_{3}||u-f^{2}u||+c_{4}||fu-fu||+c_{5}||fu-f^{2}u||),
				\\
				&&\varphi
				(c_{1}||u-fu||+c_{2}||u-fu||+c_{3}||u-f^{2}u||+c_{4}||fu-fu||+c_{5}||fu-f^{2}u||)),
				\end{eqnarray*}
				then, we have
		\begin{align}
			& ||\delta (u-fu)+fu-f^{2}u||  \notag \\
			& \leq c_{1}||u-fu||+c_{2}||u-fu||+c_{3}||u-f^{2}u||+c_{5}||fu-f^{2}u||)
			\label{6nn1}
		\end{align}%
		When $\delta >0$, and for any non-negative c. Let $c=\frac{1}{1+\delta }%
		<1\Longrightarrow \delta =\frac{1-c}{c}$ then inequality (\ref{6nn1})
		becomes:%
		\begin{align}
			& ||\frac{1-c}{c}(u-fu)+fu-f^{2}u||  \notag \\
			& \leq c_{1}||u-fu||+c_{2}||u-fu||+c_{3}||u-f^{2}u||+c_{5}||fu-f^{2}u||)
			\label{8nn1}
		\end{align}
		
		($f_{c}u=(1-c)u+cfu$ , \ $f_{c}^{2}u=(1-c)fu+cf^{2}u$, since $f$ is a
		convex mapping.) \textrm{%
			\begin{eqnarray*}
				||u-f_{c}u|| &=&||u-(1-c)u-cfu||=|u-u+cu-cfu||=||cu-cfu||=c||u-fu||, \\
				&\Longrightarrow & \\
				||u-f_{c}u|| &=&c||u-fu||
			\end{eqnarray*}%
		} Hence, inequality (\ref{8nn1}) could be written as:%
		\begin{align}
	 ||f_{c}u-f_{c}^{2}u|| \leq c_{1}||u-f_{c}u||+c_{2}||u-f_{c}u||+c_{3}||u-f_{c}^{2}u||+c_{5}||f_{c}u-f_{c}^{2}u||
			\label{9n1}
		\end{align}
		
		Therefore, from inequality (\ref{9n1}), we have that%
		\begin{eqnarray*}
			(1-c_{5})||f_{c}u-f_{c}^{2}u|| &\leq
			&(c_{1}+c_{2})||u-f_{c}u||+c_{3}||u-f_{c}^{2}u|| \\
			&\leq &(c_{1}+c_{2})||u-f_{c}u||+c_{3}(||u-f_{c}u||+||f_{c}u-f_{c}^{2}u||)
		\end{eqnarray*}
		
		so,%
		\begin{eqnarray}
			(1-c_{5}-c_{3})||f_{c}u-f_{c}^{2}u|| &\leq &(c_{1}+c_{2}+c_{3})||u-f_{c}u||, \label{3en}
		\end{eqnarray}
		since $c_i$ are arbitrary positive real numbers (each) less than $1$, without loss of generality we can choose our $c_3=c_4.$ Therefore, we have from the inequality above that
		\begin{eqnarray}
			||f_{c}u-f_{c}^{2}u|| &\leq &||u-f_{c}u|| \label{eqn29} \\
			&\mbox{then, we have} & \notag\\
			||f_{c}^{2}u-f_{c}^{3}u|| &\leq &||f_{c}u-f_{c}^{2}u||\leq ||u-f_{c}u|| \notag\\
			&&\vdots \notag\\
			||f_{c}^{n}u-f_{c}^{n+1}u|| &\leq & \cdots \leq ||f_{c}^{2}u-f_{c}^{3}u||\leq
			||f_{c}u-f_{c}^{2}u||\leq ||u-f_{c}u|| \notag 
		\end{eqnarray}%
		so, $u_n$ is a monotone decreasing, then there is an $r \in \mathbb{R}$ such that%
		\begin{equation*}
			||f_{c}^{n}u-f_{c}^{n+1}u||\rightarrow r\quad \text{as}\quad n\rightarrow
			\infty .
		\end{equation*}
		
		From (\ref{4n}) and the above,
\begin{eqnarray*}
\psi (||f_{c}^{n}u-f_{c}^{n+1}u||) 
&\leq &G\Big[\psi \big(
c_{1}||f_c^{n-1}u-f_c^{n}u||+c_{2}||f_c^{n-1}u-f_c^{n}u||+c_{3}||\\
&&f_c^{n-1}u-f_c^{n+1}u||+c_{5}||f_c^{n}u-f_c^{n+1}u||\big), \\
&&\varphi \big(c_{1}||f_c^{n-1}u-f_c^{n}u||+c_{2}||f_c^{n-1}u-f_c^{n}u||+c_{3}||\\
&&f_c^{n-1}u-f_c^{n+1}u||+c_{5}||f_c^{n}u-f_c^{n+1}u||\big)\Big] \\
&\mbox{so, we have} & \\
\psi (r) &\leq &G \Big(\psi (r),\varphi (r)\Big)\leq
\psi (r)
\end{eqnarray*}
thus, from Definitions \ref{def:1.1} and \ref{def:1.4},\\
$\psi (r)=0,$ $\ $or $\ \varphi (r)=0,$  that is $%
r=0, $ hence,%
\begin{equation}
			||f_{c}^{n}u-f_{c}^{n+1}u||\rightarrow 0\quad \text{as}\quad n\rightarrow
			\infty . \label{eqn30}
		\end{equation}
		And, we write 
		\begin{equation*}
			\psi \big(||f_{c}^{n}u-f_{c}^{n+1}u||\big) \leq G\Big(\psi \big(||f_c^{n-1}u-f_c^{n}u||\big),\varphi \big(||f_c^{n-1}u-f_c^{n}u||\big)\Big)\leq
			\psi \big(||f_c^{n-1}u-f_c^{n}u||\big) 
		\end{equation*}
		i.e
		\begin{equation}
			\psi \big(||u_n-u_{n+1}||\big) \leq G\Big(\psi \big(||u_{n-1}-u_{n}||\big),\varphi \big(||u_{n-1}-u_{n}||\big)\Big)\leq
			\psi \big(||u_{n-1}-u_{n}||\big) \label{eqn31}
		\end{equation}
		
		Next we show that $\left\{u_n\right\}$ is a Cauchy sequence. Suppose not, then there exists $\epsilon > 0$ such that for some sub-sequences $\{u_{n_k}\}$ and $\{u_{m_k}\}$ of $\{u_n\}$ where $n_k > m_k > k,$ we have
		\begin{equation}
			||u_{n_k} - u_{m_k}|| \geq \epsilon \label{eqn33}
		\end{equation} 
		so, we can choose $n_k$ depending on $m_k$ in a way that the choosing $n_k$ is the smallest integer with $n_k > m_k$ and this in line with \ref{eqn33} gives
		\begin{equation}
			||u_{n_{k-1}}-u_{m_k}|| < \epsilon \label{eqn34}
		\end{equation}
		Then, we can say
		\begin{eqnarray}
			0<\epsilon \leq ||u_{n_k} - u_{m_k}|| &\leq& ||u_{n_k} - u_{n_{k-1}}|| + ||u_{n_{k-1}}-u_{m_k}|| \notag\\
			&<& ||u_{n_k} - u_{n_{k-1}}|| + \epsilon. \label{eqn35}
		\end{eqnarray}
		Taking limit as $k \to \infty$ together with \ref{eqn30} and Lemma (1.10), we have
		\begin{equation}
			\lim_{k \to \infty}||u_{n_k} - u_{m_k}|| = \epsilon, \label{eqn36}
		\end{equation}
		and
		\begin{eqnarray}
			||u_{n_k} - u_{m_k}|| &\leq& ||u_{n_k} - u_{n_{k-1}}|| + ||u_{n_{k-1}}- u_{m_{k-1}}|| + ||u_{m_{k-1}}- u_{m_{k}}|| \notag \\
			&\leq& ||u_{n_k} - u_{n_{k-1}}|| + ||u_{n_{k-1}}- u_{n_{k}}|| + ||u_{n_{k}}- u_{m_{k-1}}|| + ||u_{m_{k-1}}-u_{m_k}|| \notag \\
			&\leq& ||u_{n_k} - u_{n_{k-1}}|| + ||u_{n_{k-1}}- u_{n_{k}}|| + ||u_{n_{k}}- u_{m_k}||\notag \\
			&+& ||u_{m_k}-u_{m_{k-1}}|| + ||u_{m_{k-1}}- u_{m_{k}}||,
		\end{eqnarray}
		using \ref{eqn30} and \ref{eqn36} as $k \to \infty,$ we have
		\begin{equation}
			\epsilon = \lim_{k \to \infty} ||u_{n_{k-1}}- u_{m_{k-1}}||. \label{eqn40}
		\end{equation}
		Also, using \ref{eqn29}, we can write 
		\begin{equation}
			||u_n - u_{n+1}|| \leq ||u_{n-1}-u_n|| \label{eqn37}
		\end{equation}
		Therefore, using (\ref{eqn37}) in (\ref{4n}) with $u=u_{n_{k}}$ and $v=u_{m_{k}},$ together with (\ref{eqn31}), gives
		\begin{eqnarray}
			\psi (||u_{n_k}-u_{m_k}||) &\leq& G\Big(\psi (||u_{n_k}-u_{m_k}||),\varphi (||u_{n_k}-u_{m_k}||)\Big)\notag \\
			&\leq& G\Big(\psi (||u_{n_{k-1}}- u_{m_{k-1}}||),\varphi (||u_{n_{k-1}}- u_{m_{k-1}}||)\Big)\notag \\
			&\leq& \psi (||u_{n_{k-1}}- u_{m_{k-1}}||)
		\end{eqnarray}
		taking limit as $k \to \infty,$ and using equations \ref{eqn36} and \ref{eqn40}, we have
		\begin{equation}
			\psi(\epsilon) \leq G\Big(\psi(\epsilon), \varphi (\epsilon)\Big) \leq \psi(\epsilon)
		\end{equation}
		thus, $\psi (\epsilon)=0,$ $\ $or $\ \varphi (\epsilon)=0,$ since $\psi$ is altering distance function, we get a contradiction.\\
		Hence, sequence $\{u_n\}$ is Cauchy and so converges to certain $u^* \in B.$
		
		\par We claim that $u^*=f_cu^*.$ Then using triangle inequality, Definition \ref{def:1.32024+328+4} together with (\ref{6nn1}) and (\ref{9n1}), we have 
		\begin{eqnarray*}
			\psi(||u^*-f_cu^*||) &\leq& \psi(||u^*-f_c^{n+1}u|| + ||f_c^{n+1}u-f_cu^*||)\\
			& =& \psi(||u^*-f_c^{n+1}u||) + \psi(||f_c^{n+1}u-f_cu^*||)\\
			&\leq& \psi(||u^*-f_c^{n+1}u||) +G\Big[\psi \Big(c_1||f_c^{n}u-u^*||+c_2||f_c^nu-f_c^{n+1}u||\\
			&+& c_3||f_c^nu-f_cu^*||+c_4||u^*-f_c^{n+1}u||+c_5||u^*-f_cu^*||\Big),\\
			&&\varphi \Big(c_1||f_c^{n}u-u^*||+c_2||f_c^nu-f_c^{n+1}u||+c_3||f_c^nu-f_cu^*||\\
			&+&c_4||u^*-f_c^{n+1}u||+c_5||u^*-f_cu^*||\Big)\Big]\\
			&&\mbox{taking limit as}~~ n \to \infty \\
			&=& \psi(||u^*-u^*||) +G\Big[\psi \Big(c_1||u^*-u^*||+c_2||u^*-u^*||\\
			&+& c_3||u^*-f_cu^*||+c_4||u^*-u^*||+c_5||u^*-f_cu^*||\Big),\\
			&&\varphi \Big(c_1||u^*-u^*||+c_2||u^*-u^*||+c_3||u^*-f_cu^*||\\
			&+&c_4||u^*-u^*||+c_5||u^*-f_cu^*||\Big)\Big]\\
			&=& G\Big[\psi \Big((c_3+c_5)||u^*-f_cu^*||\Big),\varphi \Big((c_3+c_5)||u^*-f_cu^*||\Big)\Big]\\
			&\leq&  \psi \Big((c_3+c_5)||u^*-f_cu^*||\Big),
		\end{eqnarray*}
		that is 
		$$\psi(||u^*-f_cu^*||) \leq  \psi \Big((c_3+c_5)||u^*-f_cu^*||\Big),$$
		since $\psi$ is a altering distance function and clearly $c_3+c_5 < 1,$ then we conclude that 
		$$u^*=f_cu^*.$$
		And by Lemma \ref{lem:1.32024+328+4}, we have $$u^*=f_cu^*=fu^*.$$ 
		For the uniqueness part (we prove by contradiction). Suppose not, then\\ $Fix_{f_c}=\{u^*, v^*\}$ i.e $u^* \neq v^*$ and $f_cu^*=u^*;~~f_cv^*=v^*.$\\
		We have by Definition \ref{def:1.32024+328+4} that
		\begin{eqnarray*}
			0&<&\psi(||u^*-v^*||)= \psi(||f_cu^*-f_cv^*||)\\
			&\leq& G\Big[\Big(\psi(c_1||u^*-v^*||+c_2||u^*-f_cu^*||+c_3||u^*-f_cv^*||+c_4||v^*-f_cu^*||+c_5||v^*-f_cv^*||\Big),\\
			&&\Big(\varphi (c_1||u^*-v^*||+c_2||u^*-f_cu^*||+c_3||u^*-f_cv^*||+c_4||v^*-f_cu^*||+c_5||v^*-f_cv^*||\Big)\Big]\\
			&=& G\Big[\Big(\psi(c_1||u^*-v^*||+c_2||u^*-u^*||+c_3||u^*-v^*||+c_4||v^*-u^*||+c_5||v^*-v^*||\Big),\\
			&&\Big(\varphi (c_1||u^*-v^*||+c_2||u^*-u^*||+c_3||u^*-v^*||+c_4||v^*-u^*||+c_5||v^*-v^*||\Big)\Big]\\
			&=& G\Big[\psi \Big((c_1+c_3+c_4)||u^*-v^*||\Big),\varphi \Big((c_1+c_3+c_4)||u^*-v^*||\Big)\Big]\\
			&\leq&  \psi \Big((c_1+c_3+c_4)||u^*-v^*||\Big),
		\end{eqnarray*}
		but $c_1+c_3+c_4 < 1.$ Therefore, we conclude that $u^* = v^*$\\
		~~\\
		\par \textbf{CASE 2:} $\delta = 0,$ here $f_c = f$ and (\ref{9n1}) becomes
		\begin{align}
			& ||fu-f^{2}u||\leq c_{1}||u-fu||+c_{2}||u-fu||+c_{3}||u-f^{2}u||+c_{5}||fu-f^{2}u||,
		\end{align}
		while inequalities \ref{eqn30} and \ref{eqn31} respectively becoming
		\begin{equation}
			||f^{n}u-f^{n+1}u||\rightarrow 0\quad \text{as}\quad n\rightarrow
			\infty,
		\end{equation}
		\begin{equation}
			\psi \big(||u_n-u_{n+1}||\big) \leq G\Big(\psi \big(||u_{n-1}-u_{n}||\big),\varphi \big(||u_{n-1}-u_{n}||||\big)\Big)\leq
			\psi \big(||u_{n-1}-u_{n}||\big)
		\end{equation}
		the rest of the prove is similar to that of Case 1.\\
		Hence, we have $fu^* = u^*,$ this together with the conclusion of Case 1 gives $$f_cu^*=fu^*=u^*.$$
	\end{proof}
	\begin{theorem}
		\label{t24+1th1}Let $(B,||.||)$ be a real Banach space and $S,f:B\rightarrow B$
		be commuting convex mappings satisfying definition \ref{def:1.32024+328+5}, with $f(B)\subset S(B).$ Then, $S$ and $f$ have a unique common fixed point (say $u^{\ast }$) in B.
	\end{theorem}
	
	\begin{proof}
		For a generic point $u_{0}$ in B, we define the sequence $\{Su_n\}_{n\in \mathbb{N}}\Big(\{fu_{n}\}=\{Su_{n+1}\}\Big)$ by Jungck-Schaefer iteration 
		$$Su_{n+1}=(1-c)Su_n+cfu_n~\forall~ n\geq 0.$$
		And let $fv=Su$, for any $v>u$. Using definition (\ref{def:1.32024+328+5}), we have
		\textrm{%
			\begin{eqnarray*}
				&&\psi (||\delta (Su-Sv)+fu-fv||) \\
				&\leq &G\Big(\psi
				(c_{1}||Su-Sv||+c_{2}||Su-fu||+c_{3}||Su-fv||+c_{4}||Sv-fu||+c_{5}||Sv-fv||),
				\\
				&&\varphi
				(\psi
				(c_{1}||Su-Sv||+c_{2}||Su-fu||+c_{3}||Su-fv||+c_{4}||Sv-fu||+c_{5}||Sv-fv||)\Big)
				\\
				&= &G\Big(\psi
				(c_{1}||Su-Sv||+c_{2}||fv-fu||+c_{3}||fv-fv||+c_{4}||Sv-fu||+c_{5}||Sv-Su||),
				\\
				&&\varphi
				(\psi
				(c_{1}||Su-Sv||+c_{2}||fv-fu||+c_{3}||fv-fv||+c_{4}||Sv-fu||+c_{5}||Sv-Su||)\Big)
				\\
				&=&G\Big(\psi
				((c_{1}+c_{5})||Su-Sv||+c_{4}||Sv-fu||+c_{2}||fu-fv||),
				\\
				&&\varphi
				((c_{1}+c_{5})||Su-Sv||+c_{4}||Sv-fu||+c_{2}||fu-fv||)\Big)
				\\
				&\leq&G\Big(\psi
				((c_{1}+c_{5})||Su-Sv||+c_{4}(||Su-Sv||+||Su-fu||)+c_{2}||fu-fv||),
				\\
				&&\varphi
				((c_{1}+c_{5})||Su-Sv||+c_{4}(||Su-Sv||+||Su-fu||)+c_{2}||fu-fv||)\Big)
				\\
				&=&G\Big(\psi
				((c_{1}+c_{4}+c_5)||Su-Sv||+c_{4}||fv-fu||+c_{2}||fu-fv||),
				\\
				&&\varphi
				((c_{1}+c_{4}+c_5)||Su-Sv||+c_{4}||fv-fu||+c_{2}||fu-fv||)\Big)
				\\
				&=&G\Big(\psi
				((c_{1}+c_{4}+c_{5})||Su-Sv||+(c_2+c_{4})||fu-fv||),
				\\
				&&\varphi
				((c_{1}+c_{4}+c_{5})||Su-Sv||+(c_2+c_{4})||fu-fv||)\Big)
				\\
				&\Longrightarrow &
			\end{eqnarray*}%
		}%
		\begin{align}
			||\delta (Su-Sv)+fu-fv||   \leq (c_{1}+c_{4}+c_{5})||Su-Sv||+(c_2+c_{4})||fu-fv||
			\label{10nn1}
		\end{align}%

		When $\delta =0$, then inequality (\ref{10nn1})
		becomes:%
		\begin{align}
			||fu-fv||  \leq (c_{1}+c_{4}+c_{5})||Su-Sv||+(c_2+c_{4})||fu-fv||
			\label{11nn1}
		\end{align}
		
		Therefore, from inequality (\ref{11nn1}), we have that%
		\begin{eqnarray*}
			(1-c_2-c_{4})||fu-fv|| &\leq
			&(c_{1}+c_{4}+c_{5})||Su-Sv||, \\
		\end{eqnarray*}
		
		so, since $c_i$ are arbitrary positive real numbers (each) less than $1$, without loss of generality we can choose our $c_3=c_4.$ Therefore, we have from the inequality above that	
		\begin{eqnarray}
			||fu-fv|| &\leq &||Su-Sv|| \notag \\
			&\Longrightarrow & \notag \\
			||f^{2}u-f^{2}v|| &\leq &||fu-fv||\leq ||Su-Sv|| \notag \\
			&& \vdots \notag \\
			&\Longrightarrow &  \notag \\
			||f^{n}u-f^{n+1}u||= ||f^{n}u-f^nv|| &\leq & \cdots \leq ||f^{3}u-f^{3}v||\leq
			||f^2u-f^{2}v|| \notag \\
			&\leq& ||fu-fu|| \leq ||Su-Sv||, \label{12nn1} 
		\end{eqnarray}
		so%
		\begin{equation*}
			||f^{n}u-f^{n+1}u||\rightarrow r\quad \text{as}\quad n\rightarrow
			\infty .
		\end{equation*}
		Using definition (\ref{def:1.32024+328+5}) and inequality (\ref{12nn1}), with $\delta =0,$ we arrive at
		$$	\psi (r) \leq G(\psi (r),\varphi (r))\leq \psi (r),$$
		thus, $\psi (r)=0,$ $\ $or $\ \varphi (r)=0,$ that is $%
		r=0, $ ~~~this implies that
		\begin{equation*}
			||f^{n}u-f^{n+1}u||\rightarrow 0\quad \text{as}\quad n\rightarrow
			\infty .
		\end{equation*}
		Next we show that $\{ f^nu_0 \}$ is a Cauchy sequence. This follows the same process as we have in the prove of Theorem 3.3.\\
		Thus, $\{ f^nu_0 \}$ is a Cauchy sequence and so converges to some $u^* \in B.$\\	
		~~\\
		When $\delta >0$, for any non-negative $c.$ Let $c=\frac{1}{1+\delta }%
		<1\Longrightarrow \delta =\frac{1-c}{c}$ then, using Jungck-Schaefer iteration we have,\\
		$$||\frac{1-c}{c}(Su-Sv)+fu-fv||  \leq (c_{1}+c_4+c_5)||Su-Sv||+(c_{2}+c_4)||fu-fv||$$
		$$||(1-c)(Su-Sv)+cfu-cfv|| \leq c(c_{1}+c_4+c_5)||Su-Sv||+c(c_{2}+c_4)||fu-fv||$$
		$$||(1-c)Su-(1-c)Sv+cfu-cfv|| \leq c(c_{1}+c_4+c_5)||Su-Sv||+c(c_{2}+c_4)||fu-fv||$$
		$$ ||(1-c)Su+cfu-(1-c)Sv+cfv||\leq c(c_{1}+c_4+c_5)||Su-Sv||+c(c_{2}+c_4)||fu-fv||$$
		
		$||f_cu-f_cv|| \leq c(c_{1}+c_4+c_5)||Su-Sv||+(c_{2}+c_4)||f_cu-f_cv||$\\
		~~\\
		Therefore
		\begin{eqnarray*}
			(1-c_2-c_4)||f_cu-f_cv|| &\leq& c(c_1+c_4+c_5)||Su-Sv|| 
		\end{eqnarray*}
		\mbox{so, since} $c,~and~c_i$ \mbox{are arbitrary positive real numbers (each) less than} $1$, \mbox{without loss of generality,}\\ \mbox{we can choose our} $c_3=c_4.$ \mbox{Therefore, we have from the inequality above that}	
		\begin{eqnarray}
			||f_cu-f_{c}v|| &\leq&||Su-Sv|| \label{11n3} 
		\end{eqnarray}
		
		so \ref{11n3} implies,%
		\begin{eqnarray*}
			||f_{c}^{2}u-f_{c}^{2}v|| &\leq &||f_{c}u-f_{c}v||\leq ||Su-Sv|| \\
			&&\vdots \\
			&\Longrightarrow & \\
			||f_{c}^{n}u-f_{c}^{n+1}u||=||f_{c}^{n}u-f_{c}^nv|| &\leq &...\leq ||f_{c}^{2}u-f_{c}^{2}v||\leq
			||f_{c}u-f_{c}v||\leq ||Su-Sv||
		\end{eqnarray*}%
		that is 
		\begin{equation} \label{424}
		||f_{c}^{n}u-f_{c}^{n+1}u|| \leq ||Su-Sv|| = ||f_cv-f_c^2v||
		\end{equation}
		so%
		\begin{equation*}
			||f_{c}^{n}u-f_{c}^{n+1}u||=||f_{c}^{n}u-f_{c}^{n}v||\rightarrow r\quad \text{as}\quad n\rightarrow
			\infty .
		\end{equation*}
		
		Using definition (\ref{def:1.32024+328+5}), in the same way, we also get
		$$\psi (r) \leq G(\psi (r),\varphi (r))\leq \psi (r)$$
		thus, $\psi (r)=0,$ $\ $or $\ \varphi (r)=0,$ that is $%
		r=0, $ this implies that
		\begin{equation*}
			||f_{c}^{n}u-f_{c}^{n+1}u||\rightarrow 0\quad \text{as}\quad n\rightarrow
			\infty .
		\end{equation*}
		\par The subsequent proof lines is similar to that of Theorem 3.3.\\
		Now, for the prove of common fixed point of $f$ and $S,$ by this, we claim that, $Su^* = f_cu^*.$ And by our previous arguments together with triangle inequality, we have:
		\begin{eqnarray*}
		&&\psi||Su^* - f_cu^*||\\
		&\leq& \psi\Big(||Su^* - f_c^2u^*|| + ||f_c^2u^*- f_cu^*||\Big)\\
		&=& \psi\Big(||Su^* - f_c^2u^*|| + ||f_cu^*- f_c(f_cu^*)||\Big)\\
		&\leq& \psi||Su^* - f_c^2u^*|| + G\Big(\psi
		(c_{1}||Su^*-S(f_cu^*)||+c_{2}||Su^*-f_cu^*||+c_{3}||Su^*-f_c(f_cu^*)||\\
		&&+c_{4}||S(f_cu^*)-f_cu^*||+c_{5}||S(f_cu^*)-f_c(f_cu^*)||),
		\\
		&&\varphi
		(\psi
		(c_{1}||Su^*-S(f_cu^*)||+c_{2}||Su^*-f_cu^*||+c_{3}||Su^*-f_c(f_cu^*)||\\
		&&+c_{4}||S(f_cu^*)-f_cu^*||+c_{5}||S(f_cu^*)-f_c(f_cu^*)||)\Big)\\
		&=&\psi||Su^* - f_c^2u^*|| + G\Big(\psi
		(c_{1}||Su^*-S(f_cu^*)||+c_{2}||Su^*-Su^*||+c_{3}||Su^*-f_c(Su^*)||\\
		&&+c_{4}||S(f_cu^*)-f_cu^*||+c_{5}||S(f_cu^*)-S(f_cu^*)||),
		\\
		&&\varphi
		(\psi
		(c_{1}||Su^*-S(f_cu^*)||+c_{2}||Su^*-Su^*||+c_{3}||Su^*-f_c(Su^*)||\\
		&&+c_{4}||S(f_cu^*)-f_cu^*||+c_{5}||S(f_cu^*)-S(f_cu^*)||)\Big)\\
		&=&\psi||Su^* - f_c^2u^*|| + G\Big(\psi
		(c_{1}||Su^*-f^2_cu^*||+c_{3}||Su^*-f^2_cu^*||+c_{4}||f^2_cu^*-Su^*||),
		\\
		&&\varphi
		(\psi
		(c_{1}||Su^*-f^2_cu^*||+c_{3}||Su^*-f^2_cu^*||+c_{4}||f^2_cu^*-Su^*||)\Big)\\
		&\leq& \psi||Su^* - f_c^2u^*|| + \psi
		(c_{1}||Su^*-f^2_cu^*||+c_{3}||Su^*-f^2_cu^*||+c_{4}||f^2_cu^*-Su^*||),
		\end{eqnarray*}
		then, we arrive at:
		\begin{eqnarray*}
			||Su^* - f_cu^*|| &\leq& (1+c_1+c_3+c_4)||Su^* - f^2_cu^*||=(1+c_1+c_3+c_4)||f_cu^* - f^2_cu^*||\\
			&\leq& (1+c_1+c_3+c_4)||Su^* -Sv^*||~using~(\ref{424}).\\
		\end{eqnarray*}
		Since $u^* = v^*,$ and by our assumption, we have
		\begin{eqnarray*}
			||Su^* - f_cu^*|| &=& ||Su^* - Su^*||\\
			&\leq& (1+c_1+c_3+c_4)||Su^* - Sv^*||=(1+c_1+c_3+c_4)||Su^* - Su^*||.
		\end{eqnarray*}
		This holds only for equality part with $c_1=c_3=c_4=0$\\
		then, we have, $Su^* = f_cu^*.$\\
		So, we have from the foregoing that,\\
		$f_cu^* = Su^* = u^*.$\\
		\end{proof}
	\section{Discussion}
	The results obtained in Theorems 4.1 and 4.2 generalized and extended some of the results contained in \cite{Ans1}, \cite{1,9,11} and some many others. 
	~~\\
	\section{Conclusion}
	\par The $C-$ class function was used in generalizing enriched Hardy-Rogers contraction and enriched Jungck-type Hardy-Rogers contraction and related fixed points and common fixed points results were established which are a generalization and extension of many recently introduced related results in literature. Finally, Our results unify, generalize and extend many known results in the literature.

	~~\\
	
		\textbf{Statements and Declarations}\\
	\textbf{Competing Interests:}
	The authors hereby declared that there are no competing interests as regards this article. The authors have no relevant financial or non-financial interests to disclose.\\

	\textbf{Data Availability Statements:}
	Not Applicable.
\end{document}